\documentclass[12pt]{amsart}
\usepackage{amsmath}
\usepackage{graphicx}
\usepackage[colorlinks=true]{hyperref}
\usepackage{color}
\usepackage{accents}
\usepackage{tabularx}
\usepackage{caption}
\newcommand{\beq}{\begin{eqnarray*}}
\newcommand{\feq}{\end{eqnarray*}}
\newcommand{\beqn}{\begin{eqnarray}}
\newcommand{\feqn}{\end{eqnarray}}

\newtheorem{theorem}{Theorem}[section]
\newtheorem{lemma}[theorem]{Lemma}

\theoremstyle{definition}

\theoremstyle{remark}

\numberwithin{equation}{section}

\newtheorem*{theorem*}{Theorem}

\begin{document}
\title[Shock formation in traffic flow models]{Shock formation in traffic flow models with nonlocal look ahead and behind flux}


\author{Yi Hu}
\address{Department of Mathematical Sciences, Georgia Southern University, Statesboro, Georgia, 30460}
\email{yihu@georgiasouthern.edu}
\author{Yongki Lee}
\email{yongkilee@georgiasouthern.edu}

\author{Shijun Zheng}
\email{szheng@georgiasouthern.edu}
\begin{abstract} 
In this work, we study a Lighthill-\textcolor{black}{Whitham}-Richard (LWR) type traffic flow model with a non-local flux. We identify a threshold condition for shock formation for traffic flow models with Arrhenius look-ahead-behind (i.e., nudging) dynamics with concave-convex flux.

\end{abstract}

\maketitle

\section{Introduction}


The aim of the present work is to investigate shock formation condition for the following Lighthill-\textcolor{black}{Whitham}-Richard (LWR) type \cite{LW55, R56} traffic flow model with a non-local flux,
\begin{equation}\label{trafficAB}
\left\{
  \begin{array}{ll}
    \partial_t u + \partial_x F(u, \bar{u}, \tilde{u}) =0, & t>0, x \in \mathbb{R}, \\
    u(0,x)=u_0 (x), &   x\in \mathbb{R},\hbox{}
  \end{array}
\right.
\end{equation}
where $u$ is the vehicle density normalized in the interval $[0,1]$. Here, the non-local flux $$F(u,\bar{u}, \tilde{u}):=u(1-u)^\gamma e^{-\bar{u}+\tilde{u}}$$ represents the density-flow relationship with $\gamma \geq 1$ (i.e. fundamental diagram, e.g., \cite{LZ11}), weighted by so called Arrehenius type  slow down $e^{-\bar{u}}$  (e.g., \cite{SK06}) and nudging $e^{-\tilde{u}}$  (e.g., \cite{KTP2021, PM2021}) factors. Here, we define the relaxation terms $\bar{u}$ and $\tilde{u}$ as
$$
\bar{u}(t,x):=\frac{1}{\gamma_a}\int^{x+\gamma_a} _{x} K_a u(t,y) \, dy,
$$
and
$$
\tilde{u}(t,x):=\frac{1}{\gamma_b}\int^{x} _{x-\gamma_b} K_b u(t,y) \, dy,
$$
where $\gamma_a$ and $\gamma_b$ are positive constant proportional to the look-ahead and behind distances, respectively. In addition to this, $K_a$ and $K_b$ represent constants proportional to positive interaction strengths.


In this work, we are interested in the shock formation condition of this non-local macroscopic traffic flow model. The shock formation conditions are established for several model equations of the form \eqref{trafficAB}.
  To put our study in a proper perspective we recall several recent works in the form of \eqref{trafficAB}:
   
$\bullet$ Concave-Look Ahead flux: $F=u(1-u)e^{-\bar{u}}$. Several finite time shock formation scenarios of solution were presented in \cite{LL11}. Later in \cite{LL15}, threshold conditions for the finite time shock formation were identified. Some careful numerical study with various look-ahead potentials was carried out in \cite{KP09}. Also, a maximum principal satisfying numerical scheme was studied in \cite{FV2020} and several numerical simulation were presented. Recently, when $\bar{u}=\int^\infty _x u(t,y) \, dy$, the authors in \cite{LT20} studied a sharp critical threshold condition on the initial data which distinguishes the global smooth solutions and finite time break-down. 

$\bullet$ Concave-Look Ahead and Behind flux: $F=u(1-u)^{-\bar{u}+\tilde{u}}$. The look behind intensification, or so called nudging, is intended to take into account the driving behavior of some drivers who actively react to vehicle distributions of one's ahead and behind.  The effect of nudging was studied very recently. In particular, recently, this nudging can increase the flow, and can have a strong stabilizing effect on traffic flow \cite{PM2021, KTP2021}. In \cite{L20}, the threshold condition for shock formation was identified.


$\bullet$ Concave-Convex-Look Ahead flux: $F=u(1-u)^\gamma e^{-\bar{u}}$, $\gamma \geq 1$. In \cite{ST20}, the effect of this ``right-skewed non-concave asymmetry'' in the flux was studied, and numerical schemes were designed. The threshold condition for finite time shock formation was studied when $\gamma=2$ in \cite{L19}.

$\bullet$ Several well-posedness of entropy weak solution results on non-local traffic flows and conservation laws can be found in \cite{AC2015, GS2016, KP2017, KPS2017, SCM20} and references therein.


In this work, we set  $\gamma=2$, and $\gamma_a = \gamma_b = K_a = K_b=1.$ In this case, we note that the global term $-\bar{u}+\tilde{u}$ in the flux can be simplified to $\int^{x+1} _{x-1} u(t,y) \, dy$. But in our proof, we handle $\bar{u}$ and $\tilde{u}$ separately not to exploit this simplification.

Also, to state our main result in this section, we introduce a quantity \textcolor{black}{with which we characterize} the initial behavior of $-\partial^2_{uu} F(u, \bar{u})$:
$$L^{u_0} (x) :=-\partial^2 _{uu}F(u, \bar{u}, \tilde{u})\bigg{|}_{t=0}= (4 -6 u_0 (x))e^{-\bar{u}(0,x)}.$$

The main result is stated as follow:

\begin{theorem}\label{thm_traffic} 
Consider \eqref{trafficAB} with $\gamma=2$, and $\gamma_a = \gamma_b = K_a = K_b=1.$ Suppose that $u_0 \in H^2$ and $0\leq u_0 (x) \leq 1$ for all $x \in \mathbb{R}$. Let $\xi(0):=argmax[u' _0 (x)]$ and $\eta(0):=argmin[u' _0 (x)]$.
If $L^{u_0} (\xi(0))$ and $L^{u_0} (\eta(0))$ are both positive and there exists a constant $\mu$ such that $0<\mu< \min[L^{u_0} (\xi(0)), L^{u_0} (\eta(0))]$, and 
\begin{equation}\label{traffic_threshold}
\sup_{x\in \mathbb{R}} [u' _0 (x)] \geq G\big{(} \mu, \lambda(\mu , \inf_{x\in \mathbb{R}} [u' _0 (x)] ), \xi(0), \eta(0)  \big{)},
\end{equation}
then $u_x$ must blow-up at some finite time.\\
Here,
$$\lambda(\mu, m ):= \frac{e}{\mu}\bigg{[}2+ \sqrt{4 +\frac{16}{27}\bigg{(}4-2\min\bigg{[}-\frac{4e}{\mu}, m\bigg{]} \bigg{)} } \bigg{]}$$
and
$$G(\mu, \lambda(\mu, m), \xi, \eta):= \lambda(\mu, m) \cdot \bigg{[}1- \exp \bigg{(}-\mu \lambda(\mu, m) \frac{27}{248e^2} (\min[L^{u_0}(\xi), L^{u_0}(\eta)]  -\mu) \bigg{)}  \bigg{]}^{-1}. $$
\end{theorem}

\textcolor{black}{\textbf{Remark.} Some remarks are in order at this point.\\
(i) We note that finite time blow-up of $u_x$ describes the generation of the traffic congestion in LWR type traffic flow models.
Thus the condition in \eqref{traffic_threshold} represents the following: if the maximum slope (i.e., $\sup_{x\in \mathbb{R}} [u' _0 (x)]$) of the traffic density is greater than the threshold function $G$ \emph{initially}, then the traffic jam occur in finite time. We should point out that the threshold function $G$ has several arguments, but it mainly depends on the \emph{minimum} \emph{initial} slope (i.e., $\inf_{x\in \mathbb{R}} [u' _0 (x)]$). Thus it can be interpreted that not only the car density behind the high traffic region but also the car density ahead of the thigh traffic region contribute to the blow-up.\\
(ii) In the proof of Theorem \ref{thm_traffic}, we want to emphasize the method and not the technicalities. For this reason we considered the simplest parameter choices with   $K_a = K_b = \gamma_a = \gamma_b =1$, but our argument in the proof works equally well for general parameters.  In the appendix section, we sketch the proof of the theorem in the case of general parameters, as well as their effects on the blow-up condition and physical interpretation.\\
(iii) In the course of the proof, we introduce numerous notations to simplify the system of differential equations that governs the slope of the traffic density. In addition to this, due to the non-local nature of \eqref{trafficAB}, several non-local terms are present, and controlling these non-local terms are essential in the proof. For readers' convenience, the notation table 
\ref{table1} is included in the appendix section.
}

$$$$
\section{Proof of theorem}

Let $F(u, \bar{u}, \tilde{u}):=u(1-u)^{2} e^{-\bar{u} + \tilde{u}}$ and $K_a = K_b = \gamma_a = \gamma_b =1$.  Letting $d:=u_x$ and applying $\partial_x$ to the first equation of \eqref{trafficAB} and expanding, we obtain

\begin{equation*}
\begin{split}
&\partial_t u + \partial_x F(u, \bar{u}, \tilde{u}) =0\\
\Rightarrow &\partial_t d + \partial_x \big{\{}   F_1 d + F_2 \bar{u}_x + F_3 \tilde{u}_x \big{\}}=0\\
\Rightarrow &\partial_t d + F_1 \partial_x d + F_{11}d^2 + F_{12}\bar{u}_x d + F_{13}\tilde{u}_x d\\
&\ \ \ \     +F_2\bar{u}_{xx} + F_{21}\bar{u}_x d + F_{22}\bar{u}^2 _x + F_{23}\bar{u}_x \tilde{u}_x\\
&\ \ \ \    +F_3 \tilde{u}_{xx} + F_{31} \tilde{u}_x d + F_{32}\tilde{u}_x \bar{u}_x + F_{33}\tilde{u}^2 _x =0.
\end{split}
\end{equation*}

Thus,
\begin{equation}\label{main1}
\begin{split}
\dot{d}&:= (\partial_t + F_1 \partial_x)d\\
&=-\big{\{} F_{11} d^2 + 2(F_{12}\bar{u}_x + F_{13}\tilde{u}_x )d + 2F_{23}\bar{u}_x \tilde{u}_x +F_{22}\bar{u}^2 _x + F_{33}\tilde{u}^2 _x + F_2 \bar{u}_{xx} + F_3 \tilde{u}_{xx} \big{\}}.
\end{split}
\end{equation}

Here, $F_1 : = \partial_u F$, $F_{11}:=\partial^2 _u F$, and $F_2 : = \partial_{\bar{u}}F$, etc. In addition to this, for notational convenience, we shall use $F(t,x)$ to mean $F(u(t,x), \bar{u}(t,x), \tilde{u}(t,x))$.

We note that $u(t,x)$ has \emph{a priori} $L^{\infty}$ bound during its life span $[0,T)$:
$$0\leq u(t,x) \leq 1, \ \ \forall x\in \mathbb{R}, \ 0\leq t <T,$$
which follows from the method of characteristics. The proof is similar to the proof of Lemma 2.3 in \cite{LL11}.

We calculate derivatives of $F$ with respect to $u$, $\bar{u}$ and $\tilde{u}$.   Straightforward calculation gives
\begin{equation*}
\begin{split}
&F_1 = (1- 4u+3u^2)e^{-\bar{u} + \tilde{u}}, \ \ F_2 = -u(1-u)^2e^{-\bar{u} + \tilde{u}}, \ \ F_3 = u(1-u)^2 e^{-\bar{u} + \tilde{u}},\\  
&F_{11} = (-4+6u)e^{-\bar{u} + \tilde{u}}, \ \ F_{22} = u(1-u)^2 e^{-\bar{u} + \tilde{u}}, \ \ F_{33}=u(1-u)^2e^{-\bar{u} + \tilde{u}},\\
&F_{12}= -(1-4u +3u^2)e^{-\bar{u} + \tilde{u}}, \ \ F_{13}=(1-4u +3u^2)e^{-\bar{u} + \tilde{u}}, \ \ F_{23}=-u(1-u)^2 e^{-\bar{u} + \tilde{u}}.
\end{split}
\end{equation*}

We then trace the evolution of the leading coefficient $F_{11}$ of \eqref{main1}. That is, we estimate $$\dot{F}_{11}=(\partial_t + F_1 \partial_x)F_{11}$$ along the characteristic:

\begin{lemma}\label{lemma2.1}
For $t>0$, it holds,
\begin{equation}\label{f11dotbound}
 -\frac{248e^2}{27}\leq \dot{F}_{11}(u, \bar{u}, \tilde{u}) \leq \frac{248e^2}{27}
\end{equation}
\end{lemma}

\begin{proof}
We first estimate the uniform bounds of $\bar{u}$, $F(u, \bar{u}, \tilde{u})$, $\bar{u}_x$, $\bar{u}_t$ , $\tilde{u}_x$ and $\tilde{u}_t$. Since $0\leq u(t,x)\leq 1$, $\bar{u} = \int^{x+1} _{x} u(t,y) \, dy$ and $\tilde{u} = \int^{x} _{x-1} u(t,y) \, dy$ have bounds
$$0\leq \bar{u} \leq 1, \ \ \ and \ \ \ 0\leq \tilde{u} \leq 1.$$
Since $F(u, \bar{u}, \tilde{u}) = u(1-u)^2 e^{-\bar{u} + \tilde{u}}$, the above bounds give
\begin{equation}\label{fbd}
0 \leq F(u, \bar{u}, \tilde{u}) \leq \frac{4e}{27}.
\end{equation}
In addition to this, $\bar{u}_x = \partial_x \big{(} \int^{x+1} _{x} u(t,y) \, dy  \big{)} = u(x+1)-u(x)$ has
$$-1 \leq \bar{u}_x \leq 1.$$
Similarly, 
$$- 1 \leq \tilde{u}_x \leq 1.$$
Finally, since
$$\bar{u}_t = \partial_t \bigg{(} \int^{x+1} _{x} u(t,y) \, dy  \bigg{)} = \int^{x+1} _{x} - \partial_y F(u, \bar{u}, \tilde{u}) \, dy = - F(u, \bar{u}, \tilde{u}) \bigg{|}^{y=x+1} _{y=x},$$
\eqref{fbd} gives
$$-\frac{4e}{27} \leq \bar{u}_t \leq \frac{4e}{27}.$$
Similarly it holds,
$$-\frac{4e}{27} \leq \tilde{u}_t \leq \frac{4e}{27}.$$
We now turn to the estimation of $\dot{F}_{11} =(\partial_t + F_1 \partial_x)F_{11}$. Consider 
\begin{equation}\label{est1}
\begin{split}
\partial_t F_{11} &= e^{-\bar{u} + \tilde{u}} \big{\{} 6u_t + (4-6u)\bar{u}_t - (4-6u)\tilde{u}_t  \big{\}}\\
&= e^{-\bar{u} + \tilde{u}} \big{\{} 6 (-F_1 u_x - F_2 \bar{u}_x - F_3 \tilde{u}_x) + (4-6u)\bar{u}_t - (4-6u)\tilde{u}_t  \big{\}},
\end{split}
\end{equation}
here, $u_t = -\partial_x F(u, \bar{u}, \tilde{u})= -F_1 u_x - F_2 \bar{u}_x - F_3 \tilde{u}_x$ was used. Also, we expand $F_1 \partial_x F_{11}$,
\begin{equation}\label{est2}
F_1 \partial_x F_{11}=F_1 e^{-\bar{u} +\tilde{u}} \big{\{} 6 u_x + (-4 +6u)(-\bar{u}_x + \tilde{u}_x) \big{\}}.
\end{equation}

By adding \eqref{est1} and \eqref{est2}, we obtain
\begin{equation*}
\begin{split}
\dot{F}_{11} &= (\partial_t + F_1 \partial_x)F_{11}\\
&=e^{-\bar{u} +\tilde{u}} \big{\{} (\bar{u}_x +\tilde{u}_x )(4F_1 -6uF_1 -6F_2) + (\bar{u}_t - \tilde{u}_t)(4-6u) \big{\}}\\
&=e^{-\bar{u} +\tilde{u}} \bigg{\{} (\bar{u}_x +\tilde{u}_x ) e^{-\bar{u} +\tilde{u}} \big{(}  -4\cdot (3u^3 -6u^2 +4u -1)  \big{)} + (\bar{u}_t - \tilde{u}_t)(4-6u) \bigg{\}}.\\
\end{split}
\end{equation*}
Note that $Range[ -4\cdot (3u^3 -6u^2 +4u -1)  ] = [0,4]$, for $0\leq u \leq 1$. The uniform bounds of $u$, $\bar{u}$, $\tilde{u}$, $\bar{u}_x$, $\tilde{u}_x$, $\bar{u}_t$,  and $\tilde{u}_t$ give the desired result.
\end{proof}

Now, define for $t\in [0,T)$,
$$M(t):=\sup_{x \in \mathbb{R}} [u_x(t,x)]=d(t, \xi(t)),$$
and
$$N(t):=\inf_{x \in \mathbb{R}} [u_x(t,x)]=d(t, \eta(t)).$$


This non-standard approach of tracing the dynamics  of $d$ along two different curves $\xi(t)$ and $\eta(t)$ originates in \cite{S68} proving wave breaking for the Whitham equation \cite{GW74}. The mappings $t \rightarrow \xi(t)$ and $t \rightarrow \eta(t)$
however, may be multi-valued so the curvess in general are not necessarily well-defined. To carry out Seliger's formal analysis, one needs to assume that the curves $\xi(t)$ and $\eta(t)$ are smooth. However, it turns out that it is impossible to guarantee the smoothness of such curves \cite{NS94}. This additional strong assumption was shown unnecessary later by the rigorous proof of Constantin and Escher \cite{CE98}. 


One can easily see that  $M(t) \geq 0$ and $N(t) \leq 0$.

We let $$\alpha : = \frac{248e^2}{27},$$ and integration over $[0,t]$ of \eqref{f11dotbound} gives
\begin{equation}\label{f11ineq}
-F_{11}(t) \geq -\alpha t - F_{11}(0), \ \ t>0.
\end{equation}
According to the assumption in Theorem \ref{thm_traffic}, we assumed that both $-F_{11}(0, \xi(0))=(4 -6u(0, \xi(0))) e^{-\bar{u}(0, \xi(0)) +\tilde{u}(0, \xi(0)) }$ and $-F_{11}(0, \eta(0))$ are  positive. For
 $$\mu \in \big{(}0, \min[-F_{11}(0, \xi(0)), -F_{11}(0, \eta(0))] \big{)},$$
and
\begin{equation}\label{t_star_2}
t^* = \frac{1}{\alpha}\big{(}  \min[-F_{11}(0, \xi(0)), -F_{11}(0, \eta(0))] -\mu \big{)},
\end{equation}
we see that from \eqref{f11ineq}, $\mu>0$ serves as a short time positive lower bound of $-F_{11}(t, \xi(t))$ and $-F_{11}(t, \eta(t))$, because
\begin{equation}\label{mu_bound}
-F_{11}(t, \xi(t)), -F_{11}(t, \eta(t)) \geq \mu >0, \ \ for \ \ t\in[0, t^*].
\end{equation}

Along $(t, \xi(t))$, \eqref{main1} can be written as
\begin{equation}\label{M_eq}
\begin{split}
\dot{M}&=-\big{\{} F_{11} M^2 + 2(F_{12}\bar{u}_x + F_{13}\tilde{u}_x )M + 2F_{23}\bar{u}_x \tilde{u}_x +F_{22}\bar{u}^2 _x + F_{33}\tilde{u}^2 _x + F_2 \bar{u}_{xx} + F_3 \tilde{u}_{xx} \big{\}}\\
&=-\big{\{} F_{11} M^2 + 2(F_{12}\bar{u}_x + F_{13}\tilde{u}_x )M + 2F_{23}\bar{u}_x \tilde{u}_x +F_{22}\bar{u}^2 _x + F_{33}\tilde{u}^2 _x \big{\}} - F_2 \bar{u}_{xx} - F_3 \tilde{u}_{xx} \\
&\geq -\big{\{} F_{11} M^2 + 2(F_{12}\bar{u}_x + F_{13}\tilde{u}_x )M + 2F_{23}\bar{u}_x \tilde{u}_x +F_{22}\bar{u}^2 _x + F_{33}\tilde{u}^2 _x \big{\}}\\
& \ \ \ \ \  \ \ \ \ \ \ \ \ \ \  \ \ \ \ \ \ \ \ \ \ \ \ \ \ \ \ \ \ \ \  \ \ \ \ \  \ \ \ \ \  \ \ \ \ \ -F_2 \cdot (N-M) - F_3 \cdot (M-N), \ \ \ \ a.e., 
\end{split}
\end{equation}
where the last inequality follows from the fact that
$$\bar{u}_{xx}=u_x (t, \xi(t)+1) - u_x (t, \xi(t))=u_x (t, \xi(t)+1) -M\geq N-M,$$
$$\tilde{u}_{xx}=u_x (t, \xi(t))  - u_x (t, \xi(t)-1) =M- u_x (t, \xi(t)-1) \leq M-N, $$
and the non-positivity of $F_2$ and $-F_3$.
Note that all coefficient functions above are evaluated at $\xi(t)$.

Also, along $(t, \eta(t))$, \eqref{main1} can be written as
\begin{equation}
\begin{split}
\dot{N}&=-\big{\{} F_{11} N^2 + 2(F_{12}\bar{u}_x + F_{13}\tilde{u}_x )N + 2F_{23}\bar{u}_x \tilde{u}_x +F_{22}\bar{u}^2 _x + F_{33}\tilde{u}^2 _x + F_2 \bar{u}_{xx} + F_3 \tilde{u}_{xx} \big{\}},
\end{split}
\end{equation}
where all coefficient functions above are evaluated at $\eta(t)$.
We consider
\begin{equation}\label{N_eq}
\begin{split}
\dot{N}&=-\big{\{} F_{11} N^2 + 2(F_{12}\bar{u}_x + F_{13}\tilde{u}_x )N + 2F_{23}\bar{u}_x \tilde{u}_x +F_{22}\bar{u}^2 _x + F_{33}\tilde{u}^2 _x \big{\}} - F_2 \bar{u}_{xx} - F_3 \tilde{u}_{xx} \\
&\geq -\big{\{} F_{11} N^2 + 2(F_{12}\bar{u}_x + F_{13}\tilde{u}_x )N + 2F_{23}\bar{u}_x \tilde{u}_x +F_{22}\bar{u}^2 _x + F_{33}\tilde{u}^2 _x \},  \ \ \ a.e.,
\end{split}
\end{equation}
where the last inequality follows from the fact that 
$$\bar{u}_{xx}=u_x(t, \eta(t) +1) - u_x (t, \eta(t))  =u_x(t, \eta(t) +1) -N \geq N-N =0,$$
$$\tilde{u}_{xx} = u_x (t, \eta(t)) - u_x (t, \eta(t)-1)=N -u_x (t, \eta(t)-1) \leq N-N=0, $$
and the non-positivity of $F_2$ and $-F_3$.

From \eqref{N_eq}, we find the lower bound of $N(t)$. \eqref{N_eq} can be written as
\begin{equation}\label{N_eq2}
\dot{N} \geq -F_{11}(N-N_1 (t))(N - N_2 (t)) \ \ a.e., 
\end{equation}
where
\begin{equation*}
\begin{split}
N_{1,2} (t)&=-\frac{1}{2F_{11}} \bigg{\{}2(F_{12} \bar{u}_x + F_{13}\tilde{u}_x)\\
&\mp \sqrt{4(F_{12}\bar{u}_x + F_{13}\tilde{u}_x)^2 -4F_{11} ( 2F_{23}\bar{u}_x \tilde{u}_x +F_{22}\bar{u}^2 _x + F_{33}\tilde{u}^2 _x) } \bigg{\}}.
\end{split}
\end{equation*}
We note that for $t \in [0, t^*]$, we have $-F_{11} \geq \mu >0$ and 
$$N_1 (t) \leq 0 \leq N_2(t),$$
because $2F_{23}\bar{u}_x \tilde{u}_x +F_{22}\bar{u}^2 _x + F_{33}\tilde{u}^2 _x \geq 0$. Indeed,
\begin{equation}\label{pos1}
2F_{23}\bar{u}_x \tilde{u}_x +F_{22}\bar{u}^2 _x + F_{33}\tilde{u}^2 _x  = u(1-u)^2 (\bar{u}_x - \tilde{u}_x)^2  e^{-\bar{u} + \tilde{u}}\geq 0.
\end{equation}


We bound $N_1(t)$ from below:
\begin{lemma}
$N_1(t)$ is uniformly bounded from below, in particular,
it holds
\begin{equation}\label{N_1_bound}
N_1 (t) \geq -\frac{4e}{\mu}.
\end{equation}
\end{lemma}
\begin{proof}
Let $v:=\bar{u}_x = u(x+1)-u(x)$, $w:=\tilde{u}_x = u(x)-u(x-1)$ and $h:=w-v$. Then $N_1$ can be written as
\begin{equation*}
\begin{split}
N_1 = \frac{1}{(4-6u)} \bigg{[} &-(1-4u +3u^2)(w-v)\\
& - \sqrt{\{ (1-4u +3u^2)(w-v) \}^2 -(-4 +6u)u(1-u)^2 (w-v)^2} \bigg{]}.
\end{split}
\end{equation*}
We let 
$$Q:=  -(1-4u +3u^2)h - \sqrt{\{ (1-4u +3u^2)h \}^2 -(-4 +6u)u(1-u)^2 h^2}. $$
By completing the square, 
\begin{equation*}
\begin{split}
Q^2 &= -(-4 +6u)u(1-u)^2 h^2 -2Q(1-4u +3u^2)h\\
&\leq -(-4 +6u)u(1-u)^2 h^2 + \epsilon Q^2 + \frac{(1-4u +3u^2)^2 h^2}{\epsilon}, \ \ 0<\epsilon<1.
\end{split}
\end{equation*}
It follows that
\begin{equation*}
\begin{split}
(1-\epsilon)Q^2 &\leq \frac{1}{\epsilon}h^2 \bigg{\{} (1-4u +3u^2)^2 -(-4 +6u)u(1-u)^2 \epsilon  \bigg{\}}\\
&\leq \frac{4}{\epsilon},
\end{split}
\end{equation*}
because $-2 \leq h \leq 2$.
This gives
$$Q^2 \leq \frac{4}{\epsilon (1-\epsilon)}.$$
Since $\epsilon$ is arbitrary, we choose $\epsilon=\frac{1}{2}$ to get $Q^2 \leq 16$, Hence $Q \geq -4$. From \eqref{mu_bound}, we have $(4-6u)e^{-\bar{u}+\tilde{u}}\geq \mu >0$,  for $t \in [0, t^*]$, which gives
$$N_1 (t)= \frac{1}{4-6u}Q \geq \frac{-4}{4-6u} \geq \frac{-4e^{-\bar{u} +\tilde{u}}}{\mu} \geq \frac{-4e}{\mu},$$
where the last inequality follows from $0 \leq \bar{u} \leq 1$ and $0 \leq \tilde{u} \leq 1$. This completes the proof.
\end{proof}

Now, applying Lemma 3.3 in \cite{LL15} to \eqref{N_eq2} with \eqref{N_1_bound}, we obtain,
\begin{equation}
N(t) \geq \min\bigg{[} -\frac{4e}{\mu} , N(0)\bigg{]}=: \tilde{N_{0}}, \ \ \ for \ \ t \in [0, t^*].
\end{equation}
Substituting this lower bound into \eqref{M_eq}, we obtain
\begin{equation}
\begin{split}
\dot{M}&\geq -\big{\{} F_{11} M^2 + 2(F_{12}\bar{u}_x + F_{13}\tilde{u}_x )M + 2F_{23}\bar{u}_x \tilde{u}_x +F_{22}\bar{u}^2 _x + F_{33}\tilde{u}^2 _x \big{\}}-F_2 (N-M) -F_3(M-N)\\
&=-\big{\{} F_{11} M^2 + 2(F_{12}\bar{u}_x + F_{13}\tilde{u}_x )M + 2F_{23}\bar{u}_x \tilde{u}_x +F_{22}\bar{u}^2 _x + F_{33}\tilde{u}^2 _x \big{\}} -(F_3 -F_2)M + (F_3 -F_2)N\\
&\geq - F_{11}M^2 + (F_2 -F_3 -2F_{12}\bar{u}_x -2F_{13} \tilde{u}_x)M  - (2F_{23}\bar{u}_x \tilde{u}_x +F_{22}\bar{u}^2 _x + F_{33}\tilde{u}^2 _x) -(F_2 - F_3)\tilde{N_0}
\end{split}
\end{equation}
or
\begin{equation}\label{M_eq1}
\dot{M} \geq - F_{11}(M-M_1)(M-M_2) \ \ \ \ a.e.,
\end{equation}
for $t\in[0, t^*]$. Here,
\begin{equation}\label{M12}
\begin{split}
M_{1,2}:=&\frac{1}{-2F_{11}}\bigg{\{}  -(F_2 -F_3 -2F_{12}\bar{u}_x -2F_{13}\tilde{u}_x) \mp\\
 &\sqrt{(F_2 -F_3 -2F_{12}\bar{u}_x -2F_{13}\tilde{u}_x)^2-4F_{11}\big{(} 2F_{23}\bar{u}_x \tilde{u}_x +F_{22}\bar{u}^2 _x + F_{33}\tilde{u}^2 _x +(F_2 -F_3)\tilde{N}_0 \big{)}} \bigg{\}}.
\end{split}
\end{equation}
We note that $M_1 \leq 0 \leq M_2$ because
$$-4F_{11} >0, \ and \  \big{(} 2F_{23}\bar{u}_x \tilde{u}_x +F_{22}\bar{u}^2 _x + F_{33}\tilde{u}^2 _x +(F_2 -F_3)\tilde{N}_0 \big{)} \geq 0,$$
for $t\in[0, t^*]$. Indeed, $F_2 -F_3 \leq 0$, $\tilde{N}_0 <0$ and non-negativity of $2F_{23}\bar{u}_x \tilde{u}_x +F_{22}\bar{u}^2 _x  +F_{33}\tilde{u}^2 _x$ follows from \eqref{pos1}.

We claim that $M_2$ has an uniform upper bound, that is, $M_2 (t) \leq R_M$:
\begin{lemma}\label{lemma2.3}
$R_M$ holds
$$M_2(t) \leq \frac{e}{\mu} \bigg{[} 2+ \sqrt{4+ \frac{16}{27} \big{\{} 4 - 2 \tilde{N}_0 \big{\}} } \bigg{]}=:R_M.$$
\end{lemma}
\begin{proof}
We first expand and simplify $M_2$ in \eqref{M12}: 
\begin{equation*}
\begin{split}
M_2 &= \frac{1}{-2(-4 + 6u)} \bigg{[}  P + \sqrt{P^2 -4(-4 +6u) \big{\{}  u(1-u)^2 (v-w)^2 -2u(1-u)^2 \tilde{N_0} \big{\}} }\bigg{]}\\
&=\frac{1}{-2(-4 + 6u)} \bigg{[}  P + \sqrt{P^2 -4(-4 +6u)u(1-u)^2 \big{\{}   (v-w)^2 -2 \tilde{N_0} \big{\}} }\bigg{]}.
\end{split}
\end{equation*}
Here, 
$$v:=\bar{u}_x, \ \ w:=\tilde{u}_x, \ \ and \ \ P(u,v,w):=2u(1-u)^2 -2(1-4u +3u^2)(v-w).$$
Since $v,w \in [-1,1]$, and $u \in [0,1]$, we obtain that 
$$\max\{ P(u,v,w)  \}=4.$$
Thus,
$$M_2 (t) \leq \frac{1}{-2(-4 + 6u)} \bigg{[}  4 + \sqrt{4^2 +16\cdot \frac{4}{27} \big{\{}   2^2 -2 \tilde{N_0} \big{\}} }\bigg{]}.  $$
From \eqref{mu_bound}, we have that $(4-6u)e^{-\bar{u} + \tilde{u}}\geq \mu >0$,  for $t \in [0, t^*]$, which gives
\begin{equation*}
\begin{split}
M_2 (t) &\leq \frac{e^{-\bar{u} + \tilde{u}}}{2\mu}\bigg{[}  4 + \sqrt{4^2 +16\cdot \frac{4}{27} \big{\{}   2^2 -2 \tilde{N_0} \big{\}} }\bigg{]}\\
&\leq \frac{e}{\mu} \bigg{[} 2+ \sqrt{4+ \frac{16}{27} \big{\{} 4 - 2 \tilde{N}_0 \big{\}} } \bigg{]},
\end{split}
\end{equation*}
where the last inequality follows from the fact that $0 \leq \bar{u} \leq 1$ and $0 \leq \tilde{u} \leq 1$. This completes the proof.
\end{proof}

Now, by Lemma 3.3 in \cite{LL15}, if 
$$M(0)>R_{M},$$
then $M(t)$ will blow-up in a finite time. However,  we will need to find the blow-up condition of \eqref{M_eq1} that leads to the blow-up of $M$ \emph{before} $t^*$, which defined in \eqref{t_star_2}, because \eqref{M_eq1} may not be valid after $t^*$. Finding the ``before $t^*$" blow-up condition of \eqref{M_eq1} and blow-up time estimation are carried out by comparison with the following equation:
\begin{equation}\label{alpha_eq} 
\frac{d \beta}{dt} = \mu \beta(t)(\beta(t) -R_{M}).
\end{equation}
\begin{lemma}(Lemma 3.2 in \cite{L19})\label{comp_lemma}
$$R_{M}\leq \beta(0) < M(0) \ implies \ that \ R_{M}\leq \beta(t) < M(t),$$
$\forall t \in [0, t^*]$.
\end{lemma}

Let us consider finite time blow-up of $\beta(t)$ in \eqref{alpha_eq}. By solving \eqref{alpha_eq}  directly, we obtain 
\begin{lemma}\label{bt_lemma} (Lemma 2.4 in \cite{LL15-1})
If $\beta(0) >R_M$ in \eqref{alpha_eq}, then
$$\beta(t) \rightarrow \infty \ \ as \ \ t \rightarrow \frac{1}{\mu R_M}\log \bigg{(} \frac{\beta(0)}{\beta(0) -R_M} \bigg{)}. $$
\end{lemma}

The last step of proving Theorem \ref{thm_traffic} is to combine the comparison principle in Lemma \ref{comp_lemma} with Lemma \ref{bt_lemma}. We let $\kappa=\mu R_M t^* >0$. Then for any given initial data
\begin{equation}\label{thm_condition}
M(0) > R_M \cdot \frac{e^\kappa}{e^\kappa -1} >R_M,
\end{equation}
we can always find the initial data $\beta(0)$ for \eqref{alpha_eq} such that
$$\beta(0) \in \bigg{[} \frac{R_M e^\kappa}{e^\kappa -1} ,M(0) \bigg{)}.$$
Since $\beta(0) \geq  \frac{R_M e^\kappa}{e^\kappa -1} > R_{M}$, $\beta(0)$ will lead to finite time blow-up of $\beta(t)$, which in turn implies finite time blow-up of $M(t)$ by Lemma \ref{comp_lemma}.  We notice that the finite time blow-up occur before 
\begin{equation}\label{t_star}
t^*= \frac{1}{\alpha} \bigg{(} \min[-F_{11}(0, \xi(0)) , -F_{11} (0, \eta(0))] -\mu\bigg{)}.
\end{equation}
Indeed, the blow-up time in Lemma \ref{bt_lemma} can be estimated as follows. Since $\beta(0)\geq R_M e^{\kappa}/(e^\kappa -1)$,
\begin{equation}
\begin{split}
\frac{1}{\mu R_{M}} \log\bigg{(} \frac{\beta(0)}{\beta(0) - R_{M}} \bigg{)} &\leq \frac{1}{\mu R_{M}} \log\bigg{(} \frac{ R_M e^{\kappa}/(e^\kappa -1) }{R_M e^{\kappa}/(e^\kappa -1)  - R_{M}} \bigg{)}\\
&=\frac{1}{\mu R_{M}} \kappa=t^*.
\end{split}
\end{equation}
We now have $\beta(t) \rightarrow \infty$ before $t^*$, which in turn by Lemma \ref{comp_lemma}, implies that $M(t)\rightarrow \infty$ in finite time.  \textcolor{black}{By expanding \eqref{thm_condition}, we obtain the condition in Theorem \ref{thm_traffic}. Indeed,
\begin{equation*}
\begin{split}
M(0) &> R_M \cdot \frac{e^{\kappa}}{e^{\kappa} -1}\\
&=R_M (1-e^{-\kappa})^{-1}\\
&=R_M (1-e^{-\mu R_M t^{*}})^{-1}.
\end{split}
\end{equation*}
By substituting the $R_M$ expression in Lemma \ref{lemma2.3} and $t^*$ in \eqref{t_star} into the above expression, respectively, we obtain \eqref{traffic_threshold} in Theorem \ref{thm_traffic}. }
This completes the proof of Theorem \ref{thm_traffic}.

\section{Appendix}
\subsection{Notation table}
In this sub-section, we display the table which includes most of notations introduced in the paper are defined.

\begin{center}

\begin{tabular}{||>{\color{black}}c | >{\color{black}}m{30em}  ||} 
 \hline
 \textbf{Symbol} & \textbf{Expression and/or the place where it is introduced in}   \\ [0.5ex] 
 \hline\hline
  $d$& derivative of $u(t,x)$ with respect to $x$, Page 3\\ 
\hline
 $F_i$ and $F_{ij}$& derivatives of $F(u, \bar{u}, \tilde{u})$ with respect to its argument(s), Page 3\\ 
\hline
$\bar{u}(t,x)$& $\int^{x+1} _x u(t,y) dy$, in Lemma \ref{lemma2.1}, Page 4\\ 
\hline
$\tilde{u}(t,x)$& $\int^{x} _{x-1} u(t,y) dy$, in  Lemma \ref{lemma2.1}, Page 4\\ 
\hline
$M(t)$& $\sup_{x\in \mathbb{R}}[u_x (t,x)]$, attained along $\xi(t)$, Page 5\\ 
\hline
$N(t)$& $\inf_{x\in \mathbb{R}}[u_x (t,x)]$, attained along $\eta(t)$,  Page 5\\ 
\hline
$\alpha$& $\frac{248e^2}{27}$,  Page 5\\ 
\hline
$\mu$ &   Page 5\\ 
\hline
$t^*$ & $\frac{1}{\alpha}\big{(}  \min[-F_{11}(0, \xi(0)), -F_{11}(0, \eta(0))] -\mu \big{)}$, equation \eqref{t_star_2},  Page 5\\ 
\hline
$N_1 (t)$ and $N_2 (t)$ & Roots of the quadratic equation in \eqref{N_eq}  or  \eqref{N_eq2}, Page 6\\ 
\hline
$M_1 (t)$ and $M_2 (t)$ & Roots of the quadratic equation in \eqref{M_eq}  or \eqref{M_eq1}, Page 7\\ 
\hline
$\tilde{N_{0}}$ & $\min\big{[} -\frac{4e}{\mu} , N(0)\big{]}$, Page 7\\ 
 \hline
  $R_M$ & upper bound of $M_2 (t)$,  Lemma \ref{lemma2.3}, Page 8\\ 
\hline
 $\kappa$ & $\mu R_M t^* $,  Lemma \ref{bt_lemma}, Page 9\\ 

 \hline
\end{tabular}
\captionof{table}{\textcolor{black}{Notation table}}\label{table1}
\end{center}

\textcolor{black}{
\subsection{Discussions on general parameters} 
In this subsection, we briefly sketch the proof of the theorem in the case of general parameters $K_a, K_b, \gamma_a$ and $\gamma_b$, as well as their effects on the blow-up condition. From
$$\bar{u}(t,x)=\frac{1}{\gamma_a}\int^{x+\gamma_a} _x K_a u(t,y) \, dy  \  \ and \ \ \tilde{u}(t,x)=\frac{1}{\gamma_b}\int^{x} _{x-\gamma_b} K_b u(t,y) \, dy,$$
we obtain
$$0\leq \bar{u} \leq K_a \ \ and \ \ 0\leq \tilde{u} \leq K_b,$$
respectively.
Since $F(u, \bar{u}, \tilde{u}) = u(1-u)^2 e^{-\bar{u} + \tilde{u}}$, and $0\leq u \leq 1$, the above bounds give
\begin{equation}\label{appen_F_bound}
0 \leq F(u, \bar{u}, \tilde{u}) \leq \frac{4}{27}e^{K_b}.
\end{equation}
We note that the above estimate corresponds to \eqref{fbd}. Also, from
$$\bar{u}_x = \frac{K_1}{\gamma_a} \big{(} u(x+\gamma_a) -u(x) \big{)} \  \ and \ \ \tilde{u}_x = \frac{K_b}{\gamma_b} \big{(} u(x)- u(x-\gamma_b) ),$$
$$-\frac{K_a}{\gamma_a} \leq \bar{u}_x \leq \frac{K_a}{\gamma_a} \ \ and \ \ -\frac{K_b}{\gamma_b} \leq \tilde{u}_x \leq \frac{K_b}{\gamma_b}$$
are derived.
Finally, since
$$\bar{u}_t = \partial_t \bigg{(} \frac{1}{\gamma_a}\int^{x+\gamma_a} _{x} K_a u(t,y) \, dy  \bigg{)} = \frac{K_a}{\gamma_a}\int^{x+\gamma_a} _{x} - \partial_y F(u, \bar{u}, \tilde{u}) \, dy = - \frac{K_a}{\gamma_a}F(u, \bar{u}, \tilde{u}) \bigg{|}^{y=x+\gamma_a} _{y=x},$$
\eqref{appen_F_bound} gives
$$-\frac{4K_a}{27\gamma_a}e^{K_a} \leq \bar{u}_t \leq \frac{4K_a}{27\gamma_a}e^{K_b}.$$
Similarly, it holds
$$-\frac{4K_b}{27\gamma_b}e^{K_b} \leq \tilde{u}_t \leq \frac{4K_b}{27\gamma_b}e^{K_b}.$$
For the estimation of $\dot{F}_{11} =(\partial_t + F_1 \partial_x)F_{11}$, we reuse the expansion in the proof of Lemma \ref{lemma2.1}, then apply the uniform bounds of $u$, $\bar{u}$, $\tilde{u}$, $\bar{u}_x$, $\tilde{u}_x$, $\bar{u}_t$,  and $\tilde{u}_t$, derived above. Indeed,
\begin{equation*}
\begin{split}
\dot{F}_{11} &= (\partial_t + F_1 \partial_x)F_{11}\\
&=e^{-\bar{u} +\tilde{u}} \bigg{\{} (\bar{u}_x +\tilde{u}_x ) e^{-\bar{u} +\tilde{u}} \big{(}  -4\cdot (3u^3 -6u^2 +4u -1)  \big{)} + (\bar{u}_t - \tilde{u}_t)(4-6u) \bigg{\}}\\
\end{split}
\end{equation*}
implies
\begin{equation}\label{general}
\begin{split}
|\dot{F}_{11}|&\leq e^{K_b} \bigg{|}\big{(} \frac{K_a}{\gamma_a} + \frac{K_b}{\gamma_b} \big{)}e^{K_b} \big{(}  -4\cdot (3u^3 -6u^2 +4u -1)  \big{)}
+\big{(} \frac{4K_a}{27\gamma_a}e^{K_b} + \frac{4K_b}{27\gamma_b}e^{K_b} \big{)}\cdot 4 \bigg{|}.\\
&\leq e^{2K_b} \frac{124}{27} \bigg{(} \frac{K_a}{\gamma_a} + \frac{K_b}{\gamma_b} \bigg{)}
\end{split}
\end{equation}
where the last inequality holds due to  $Range[ -4\cdot (3u^3 -6u^2 +4u -1)  ] = [0,4]$, for $0\leq u \leq 1$. We note that if $K_a = K_b=\gamma_a = \gamma_b =1$, then the above estimate is reduced to \eqref{f11dotbound} in Lemma \ref{lemma2.1}. The remaining steps are straightforward, so the detail will be left to the readers.
}
\textcolor{black}{
In the proof of the main theorem, we set $\alpha$ as the bound of $\dot{F}_{11}$ (i.e., $\alpha=\frac{248}{27}e^2$, when $K_a = K_b=\gamma_a = \gamma_b =1$). In the statement of the main theorem, one can see that $\alpha$ acts as a blow-up threshold amplifying term.  Indeed, consider the blow-up threshold in the main theorem,
$$G(\mu, \lambda(\mu, m), \xi, \eta):= \lambda(\mu, m) \cdot \bigg{[}1- \exp \bigg{(}-\mu \lambda(\mu, m) \underbrace{\frac{27}{248e^2} }_{=1/\alpha}(\min[L^{u_0}(\xi), L^{u_0}(\eta)]  -\mu) \bigg{)}  \bigg{]}^{-1}. $$
We see that $G$ is increasing in $\alpha$. Thus, heuristically speaking, the bigger $\alpha$ means the less chance of blow-up.
}

\textcolor{black}{
Now, we let $\alpha$ be the bound in \eqref{general}, that is 
$$\alpha:= e^{2K_b} \frac{124}{27} \bigg{(} \frac{K_a}{\gamma_a} + \frac{K_b}{\gamma_b} \bigg{)}.$$
Following the previous interpretation regarding the effect of $\alpha$ on the blow-up threshold, we observe that  the higher values of $K_a$ and $K_b$ (correspond to stronger interaction  between vehicles) increase the blow-up threshold,  whereas the higher values of $\gamma_a$ and $\gamma_b$ (correspond to longer look-ahead and behind distance, respectively) suppress the blow-up threshold. One may interpret this as follows:  a group of drivers who carefully react to nearby vehicles experience less traffic jam, on the other hand, a group of drivers who take into account too long distance (thus putting relatively less attention to nearby vehicles) may experience more traffic jam. 
 }



\bibliographystyle{abbrv}

\end{document}